\documentclass{amsart}
\usepackage{pgf,tikz}
\usepackage{amssymb}
\usepackage{hyperref}
\usepackage{pst-node}
\usepackage{tikz-cd}
\usepackage[utf8]{inputenc}
\usepackage{pgf,tikz}
\usepackage{mathrsfs}
\usetikzlibrary{arrows}
\usepackage{pdfpages}
\usepackage{xcolor}

\usepackage{graphicx}

\theoremstyle{plain}

\newtheorem{mainthm}{Theorem}

\newtheorem*{conj*}{Conjecture}
\newtheorem*{cor*}{Corollary}
\newtheorem*{def*}{Definition}

\newtheorem{theorem}{Theorem}[section]
\newtheorem{proposition}{Proposition}

\newtheorem{lemma}[theorem]{Lemma}

\theoremstyle{definition}

\newtheorem{example}{Example}

\theoremstyle{remark}
\newtheorem{remark}[theorem]{Remark}
\newtheorem{q}{Question}

\newcommand{\Z}{\mathbb{Z}}
\newcommand{\N}{\mathbb{N}}
\newcommand{\R}{\mathbb{R}}
\newcommand{\eps}{\varepsilon}

\newcommand{\diam}{\operatorname{diam}}

\title[Suspensions with two-sided limit shadowing]{Suspensions of homeomorphisms with the two-sided limit shadowing property}

\author[Jes\'us Aponte, Bernardo Carvalho and Welington Cordeiro]{Jes\'us Aponte, Bernardo Carvalho and Welington Cordeiro}

\begin{document}

\maketitle
\begin{abstract}{In this paper we discuss the two-sided limit shadowing property for continuous flows defined in compact metric spaces. We analyze some of the results known for the case of homeomorphisms in the case of continuous flows and observe that some differences appear in this scenario. We prove that the suspension flow of a homeomorphism satisfying the two-sided limit shadowing property also satisfies it. This gives a lot of examples of flows satisfying this property, however it enlighten an important difference between the case of flows and homeomorphisms: there are flows satisfying the two-sided limit shadowing property that are not topologically mixing, while homeomorphisms satifying the two-sided limit shadowing property satisfy even the specification property. There are no homeomorphisms on the circle satisfying the two-sided limit shadowing property but we exhibit examples of flows on the circle satisfying it. It can happen that a suspension flow has the two-sided limit shadowing property but the base homeomorphism does not, though it is proved that it must satisfy a strictly weaker property called two-sided limit shadowing with a gap (as in \cite{CK}). We define a similar notion of two-sided limit shadowing with a gap for flows and prove that these notions are actually equivalent in the case of flows. Finally, we prove that singular suspension flows (in the sense of Komuro \cite{K}) do not satisfy the two-sided limit shadowing property.}
\end{abstract}



\bigskip


\section{Introduction and Statement of Results}

In topological dynamics, an important place is given to the shadowing theory, where many variants of pseudo-orbit tracing properties are discussed, mainly considering different notions of pseudo-orbits and shadowing points. Among them there is the \emph{limit shadowing property} which has been given much attention recently (see \cite{Cthesis}, \cite{C}, \cite{C2}, \cite{CC}, \cite{CK}, \cite{ENP}, \cite{KKO}, \cite{KO}, \cite{P1}, \cite{P} and others). It deals with pseudo-orbits indexed by positive integers and with one-step errors converging to zero in the future, usually called \emph{limit pseudo-orbits}, and with orbits shadowing them in the limit (see Section 2 for precise definitions). A two-sided analogue of the limit shadowing property is also defined in the literature considering pseudo-orbits indexed by the integers and with the same limit conditions both in the future and in the past. It is called the \emph{two-sided limit shadowing property}. The dynamics of systems with such property has been studied (see  \cite{Cthesis}, \cite{C}, \cite{C2} and \cite{CK}) and the class of homeomorphisms satisfying it is growing (see \cite{ACCV}, \cite{CC} and \cite{CC2}). It is known that the two-sided limit shadowing property differs in several ways from the limit shadowing property. For example, systems with the two-sided limit shadowing property must have the shadowing property, the average shadowing property, the asymptotic average shadowing property, must be topologically mixing, admit the specification property and have positive topological entropy (see \cite{CK} for the proofs) making it one of the strongest known notions of pseudo-orbit tracing properties, while there are systems with the limit shadowing property but without any of these dynamical properties (see \cite{GOP} for example). Examples of homeomorphisms satisfying the two-sided limit shadowing property: topologically mixing Anosov diffeomorphisms (and more generally topologically hyperbolic homeomorphisms \cite{C}), pseudo-Anosov diffeomorphisms of the two-dimensional sphere \cite{ACCV} and some wild examples of infinite products of subshifts \cite{CK} without periodic points. 

In this note, we consider suspension flows of homeomorphisms satisfying the two-sided limit shadowing property, analyze some of the above mentioned results known for the case of homeomorphisms in this case of suspension flows and explain the differences that appear in this scenario. Before stating the results we define the main notion of this paper: the two-sided limit shadowing property for flows. Through the whole paper we let $(X,d)$ denote a metric space.  A {\em  flow} in $X$ is a continuous function $\phi\colon X\times\mathbb{R}\to X$ satisfying 
\begin{enumerate}
\item $\phi(x,0)=x$, $\forall\,x\in X$ and
\item $\phi(\phi(x,t),s)=\phi(x,s+t)$, $\forall\,x\in X$ and $\forall\,s,t\in\mathbb{R}$. 
\end{enumerate}
Given $t\in\mathbb{R}$, the {\em time $t$-map} $\phi_t\colon X\to X$, defined by $\phi_t(x)=\phi(x,t)$ is a homemorphism of $X$ which inverse is $\phi_{-t}$. The orbit of a point $x\in X$ with respect to $\phi$ is the set $\mathcal{O}(x):=\{\phi_t(x); \,\, t\in\mathbb{R}\}$.
A sequence of pairs of points and times $(x_i,t_i)_{i\in \mathbb{Z}}\subset X\times \mathbb{R}$ is called a {\em two-sided limit pseudo-orbit} of $\phi$ if 
$$t_i\geq 1 \,\,\,\,\,\, \forall\, i\in\Z \,\,\,\,\,\, \text{and} \,\,\,\,\,\, \lim_{|i|\to\infty}d(\phi_{t_i}(x_i),\,x_{i+1})=0.$$ 
We denote by $x_0\star t$ the point in $X$ that is $t$ {\em units of time from} $x_0$ on the pseudo-orbit $(x_i,t_i)_{i\in \mathbb{Z}}$. 
Precisely, consider $(s_i)_{i\in\mathbb{Z}}$ the {\em sequence of sums of $(t_i)_{i\in\mathbb{Z}}$} defined by
\[s_i =
\begin{cases}
\displaystyle{\sum_{j=0}^{i-1}} t_j & i > 0,\\
0				             & i = 0, \\
-\displaystyle{\sum_{j=i}^{-1}} t_j & i < 0,
\end{cases}
\]
so that we can write
 \[x_0\star t = \phi_{\,t\,-\,s_i}(x_i)\ \mbox{whenever } s_i\leq t < s_{i+1}.\]

Let Rep denotes the set of all increasing homeomorphisms $f\colon\mathbb{R}\to\mathbb{R}$ satisfying $f(0)=0$. Elements of Rep are called \emph{reparametrizations} . Now we can define the two-sided limit shadowing property.

\begin{def*}
	We say that $\phi$ has the {\em two-sided limit shadowing property } if for every two-sided limit pseudo-orbit $(x_i,\,t_i)_{i\in\mathbb{Z}}$ of $\phi$ there are $h\in\mathrm{Rep}$ and $y\in X$ satisfying
	\[
	\lim_{|t|\to\infty}d(x_0\star t,\, \phi_{h(t)}(y)) = 0.
	\]
	In this case, we say that $y$  {\em two-sided limit shadows} $(x_i,\,t_i)_{i\in\mathbb{Z}}$ and that $(x_i,\,t_i)_{i\in\mathbb{Z}}$ is {\em two-sided limit shadowed} by $y$. 
\end{def*}
This property, in the set of continuous flows, was first defined in \cite{Zhu} where the authors prove that it is invariant under topological equivalences. It is also proved in \cite{Zhu} that a homeomorphism $f\colon X\to X$ of a compact metric space $X$ has the limit shadowing property if, and only if, its suspension flow $\phi^{f,\tau}$ (see precise definitions in Section 2) has the limit shadowing property, but curiously nothing is said about a similar result for the two-sided limit shadowing property, even though they consider it extensively in their paper. In the first result of this paper, we prove a two-sided analogue of their result and analyze some interesting consequences that are not clear at a first glance.  
\begin{mainthm}\label{sus}
If a homeomorphism $f\colon X\to X$ of a compact metric space $X$ satisfies the two-sided limit shadowing property, then its suspension flow $\phi^{f,\tau}$ satisfies the two-sided limit shadowing property.
\end{mainthm}
First, this proves that topological mixing (and also specification, positive entropy, average shadowing and asymptotic average shadowing) is not a necessary assumption for a flow to admit the two-sided limit shadowing property, strongly contrasting with the case of homeomorphisms (see \cite{CK}). When the base homeomorphism is transitive, then its suspension flow is also transitive but it is clearly not topologically mixing (when the height function is constant, see \cite{Bo}). Second, it is easy to see that the identity map defined on a space consisting of a single point has the two-sided limit shadowing property, so its suspension, which is conjugate to the rotation flow $\phi_t(z) = e^{2\pi i t }z$ defined in the circle $S^1$, has the two-sided limit shadowing property. This is also interesting since there are no homeomorphisms in $S^1$ satisfying the two-sided limit shadowing property (see the discussion after Theorem B in \cite{CK}) but this is an example of a continuous flow in $S^1$ satisfying it. Third, the converse of this theorem is false as the next example shows. 
\begin{example}\label{discrete}
Let $X=\{a,\,b\}$ be a set with two distinct points endowed with the discrete metric and consider the homeomorphism $f\colon X\to X$ defined by $f(a) = b$ and $f(b) = a$. It is proved in \cite{CK} that this map does not have the two-sided limit shadowing property. The suspension space of $X$ is the space $$X^f = (\{a\}\times [0,\,1]\cup \{b\}\times [0,\,1])/\sim$$ with $(b,\,1)\sim (a,\,0)$ and $(a,\, 1)\sim (b,\,0)$. It can be easily seen that $X^f$ is homeomorphic to $S^1$ and that $\phi^f$ is the rotation flow with a velocity change of factor 2. So this flow is conjugate to $\psi_t(z) = e^{\pi i t}z$ and equivalent to $\phi_t(z) = e^{2\pi i t}z$ under the equivalence $(id_{S^1}, \sigma)$ where $\sigma(z,\,t) = 2t$ and therefore admits the two-sided limit shadowing property (as observed above).
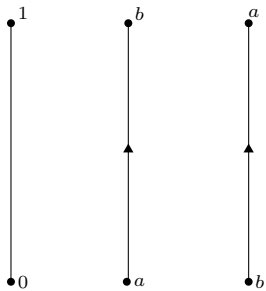
\begin{center}
		\begin{tikzpicture}[line cap=round,line join=round,>=triangle 45,x=1.0cm,y=1.0cm]
	\clip(-3.98,0.23) rectangle (1.93,4.66);
	\draw (-2.66,4.04)-- (-2.66,0.6);
	\draw (-1.1,4.04)-- (-1.1,0.6);
	\draw (0.5,4.04)-- (0.5,0.6);
	\begin{scriptsize}
	\fill [color=black] (-2.66,4.04) circle (1.5pt);
	\draw[color=black] (-2.5,4.17) node {$1$};
	\fill [color=black] (-2.66,0.6) circle (1.5pt);
	\draw[color=black] (-2.5,0.6) node {$0$};
	\fill [color=black] (-1.1,4.04) circle (1.5pt);
	\draw[color=black] (-0.95,4.17) node {$b$};
	\fill [color=black] (-1.12,0.6) circle (1.5pt);
	\draw[color=black] (-0.95,0.6) node {$a$};
	\fill [color=black] (0.5,4.04) circle (1.5pt);
	\draw[color=black] (0.57,4.17) node {$a$};
	\fill [color=black] (0.5,0.6) circle (1.5pt);
	\draw[color=black] (0.65,0.6) node {$b$};
	\fill [color=black,shift={(-1.1,2.36)}] (0,0) ++(0 pt,2.25pt) -- ++(1.95pt,-3.375pt)--++(-3.9pt,0 pt) -- ++(1.95pt,3.375pt);
	\fill [color=black,shift={(0.5,2.36)}] (0,0) ++(0 pt,2.25pt) -- ++(1.95pt,-3.375pt)--++(-3.9pt,0 pt) -- ++(1.95pt,3.375pt);
	\end{scriptsize}
	\end{tikzpicture}
\end{center}
\end{example}


With this example in mind, we obtain a characterization for the suspension flow to satisfy the two-sided limit shadowing property. 

\begin{mainthm}\label{gap}
The suspension flow $\phi^{f,\tau}$ satisfies the two-sided limit shadowing property if, and only if, its base homeomorphism $f$ satisfies the two-sided limit shadowing property with a gap.
\end{mainthm}

The base homeomorphism $f$ satisfies a weaker notion (introduced in \cite{CK}) called \emph{two-sided limit shadowing with a gap}. In this property, gaps between the shadowing orbit and the pseudo orbit are allowed to exist. The base homeomorphism in Example \ref{discrete} has the two-sided limit shadowing property with gap 1. Information about this property can be found in \cite{CK}. Now we define an analogous property in the case of continuous flows.

\begin{def*}
	We say that $\phi$ has the {\em two-sided limit shadowing property with gap $N\in\N$} if for every two-sided limit pseudo-orbit $(x_i,\,t_i)_{i\in\mathbb{Z}}$ of $\phi$, there are $h\in\mathrm{Rep}$, $y\in X$ and $K\in\R$ satisfying $|K|\leq N$,
$$d(x_0\star t, \phi_{h(t)}(z))\to0, \,\,\,\,\,\, t\to-\infty \,\,\,\,\,\, \text{and}$$
$$d(x_0\star t, \phi_{h(t)+K}(z))\to0, \,\,\,\,\,\, t\to+\infty.$$	
	In this case, we say that $y$  {\em two-sided limit shadows} $(x_i,\,t_i)_{i\in\mathbb{Z}}$ with gap $K$ and that $(x_i,\,t_i)_{i\in\mathbb{Z}}$ is {\em two-sided limit shadowed} by $y$ with gap $K$. 
\end{def*}
We prove that in the case of flows, there is no difference between this and the two-sided limit shadowing property, strongly contrasting with the case of homeomorphisms. The idea is that when a shadowing orbit has a gap in its shadowing relation we can reparametrize the flow to remove this gap.

\begin{mainthm}\label{equivalent}
A continuous flow satisfies the two-sided limit shadowing property with a gap if, and only if, it satisfies the two-sided limit shadowing property.
\end{mainthm}

Even though topological mixing is not necessary for a continuous flow to admit the two-sided limit shadowing property, transitivity is, and also is the finite shadowing property as the next theorem shows. We note that, in general, the shadowing property is not equivalent to the finite shadowing property (see \cite{K}) but this is true for flows without fixed points.



\begin{mainthm}\label{transha}
If a continuous flow has the two-sided limit shadowing property, then it is transitive and satisfies the finite shadowing property. If it does not have fixed points, then it satisfies the shadowing property.
\end{mainthm}


The following problem seems reasonable to be considered.

\begin{q}
Does there exists a continuous flow with fixed points and satisfying the two-sided limit shadowing property?
\end{q}

A singularity can be introduced in a suspension flow as Komuro did in \cite{K} defining the \emph{singular suspension flow} (see Section \ref{singu} for precise definitions). We prove this fixed point cannot coexists with the two-sided limit shadowing property.

\begin{mainthm}\label{singularsus}
The singular suspension of a homeomorphism of a compact metric space does not satisfy the two-sided limit shadowing property.
\end{mainthm}


\section{Suspensions}

We begin this section with the definition of the suspension flow and after that we prove Theorem \ref{sus}. Let $f\colon X\to X$ be a homeomorphism defined in a compact metric space and $\tau\colon X\to\R^+$ be a continuous map. Consider the suspension space
\[X^{\tau,f}:=\{(x,t):\ 0\leq t\leq \tau(x),\  x\in X\}/\sim,\]
where $(x,\tau(x))\sim(f(x),0)$ for all $x\in X$ and endow it with the usual quotient topology. For each $x\in X$, let $(s_n(x))_{n\in\Z}$ denote the sequence of sums of the sequence $(\tau(f^n(x)))_{n\in\mathbb Z}$. The {\em suspension flow of $f$ with height function $\tau$} is the flow on $X^{\tau,f}$ defined by
$$\phi^{\tau,f}_t(x,s):=(f^n(x),s+t - s_n(x))$$
whenever $s+t \in [s_n(x),\,s_{n+1}(x))$, for some $n\in\mathbb{Z}$. 
Note that when $\tau\equiv 1$, $s_n(x)=\lfloor t+s\rfloor$, the integer part of $t+s$. Every suspension of $f$ is conjugate to the suspension of $f$ under the constant function $\tau\equiv1$. A homeomorphism from $X^{1,f}$ to $X^{\tau,f}$ that conjugates the flows is given by the map $(x,t)\mapsto (x,t\tau(x))$. Since the two-sided limit shadowing property is invariant under a conjugacy, we will concentrate only in suspensions flows with height function 1. In this case, we denote the suspension space $X^{1,f}$ by $X^f$ and the suspension flow $\phi^{1,\,f}$ by $\phi^f$. We can assume that $\mathrm{diam}(X)\leq 1$ and introduce a metric in $X^f$, known as the {\em Bowen-Walters metric}, as follows. Consider the subset $X\times\{t\}$ and give it the metric $d_t$ defined by 
$$d_t((x,t),(y,t)) = (1-t)d(x,y)+td(f(x),f(y))\,\,\,\,\,\, \text{for every} \,\,\,\,\,\, x,y\in X.$$ 
For each $(x,t)$, $(y,s)\in X^f$ we consider all finite sequences $(z_i,t_i)_{i=1}^n$ of elements of $X^f$ such that $(z_1,t_1) = (x, t)$ , $(z_n,t_n)=(y,s)$ and for each $1\leq i\leq n-1$ either $(x_i,t_i)$ and $(x_{i+1},t_{i+1})$ belongs to the same set $X\times \{t\}$ (in which case we call $[(x_i,t_{i}),\,(x_{i+1},t_{i+1})]$ a \emph{horizontal segment}) or they belong to the same orbit of the suspension flow (and then we call $[(x_i,t_{i}),\,(x_{i+1},t_{i+1})]$ a \emph{vertical segment}). The sequence $(z_i,t_i)_{i=1}^n$ is called a \emph{chain}  between $(x,t)$ and $(y,s)$. The length of a horizontal segment will be given by the distance $d_t$ between the points $(x_i,t_i)$ and $(x_{i+1},t_{i+1})$, while the length of a vertical segment will be the distance between its times in $[0,1]$. The length of a chain is defined as the sum of the length of its horizontal and vertical segments. 
We define a metric $d^f$ in $X^f$ as follows:
\[d^f((x,t),(y,s)) = \inf\left\{\text{length of all chains between } (x,t)\text{ and }(y,s)\right\}.\]
It is proved in \cite{Bw} that $d^{f}$ is indeed a metric on the suspension space $X^f$ that generates the quotient topology and such that the suspension flow is continuous.
We are finally ready to proceed to the proof of Theorem \ref{sus}. This proof follows the ideas of Thomas \cite{Th} in the case of the shadowing property and we will assume that the reader knows precisely Lemmas 2.4 and 2.5 there, since we will use them repeatedly but we will not state them here.

\begin{proof}[Proof of Theorem \ref{sus}]

It is enough to consider the suspension of $f$ under the map $\tau\equiv 1$ since the two-sided limit shadowing property is invariant by conjugacies \cite{Zhu}. Suppose that $f$ has the two-sided limit shadowing property and let $((x_k,\, s_k),\,t_k)_{k\in\mathbb{Z}}$ be a two-sided limit pseudo orbit of $\phi^f$ such that $2\leq t_k<4$ for every $k\in\Z$. We will prove that it is two-sided limit shadowed and this is enough to prove the two-sided limit shadowing property of $\phi^f$ as observed in \cite{Zhu} Proposition 2.2.
For each $k\in\Z$, let $w_k=\lfloor s_k+t_k\rfloor$ denote the integer part of $s_k+t_k$. Thus, $$\phi^f_{t_k}(x_k,s_k)=(f^{w_k}(x_k),s_k+t_k-w_k)$$ and since $((x_k,s_k),t_k)_{k\in\mathbb{Z}}$ is a two-sided limit pseudo orbit, it follows that
\[\lim_{|k|\to\infty}d^{f}((f^{w_k}(x_k),s_k+t_k-w_k),\, (x_{k+1},s_{k+1}))=0.\]
Then choose $M>0$ such that $$d^{f}((f^{w_k}(x_k),s_k+t_k-w_k),\, (x_{k+1},s_{k+1}))<\frac{1}{4} \,\,\,\,\,\, \text{whenever} \,\,\,\,\,\, |k|\geq M.$$
Then by Lemma 2.4 in \cite{Th} we have that either $\vert s_k+t_k-w_k-s_{k+1}\vert<\frac{1}{4}$ or $\vert 1+s_k+t_k-w_k-s_{k+1}\vert<\frac{1}{4}$ or $\vert 1+s_{k+1}+w_k-t_k-s_{k}\vert<\frac{1}{4}$, provided $|k|\geq M$.  Now, for each $k\in\Z$ define a positive integer $n_k$ as follows: if $|k|< M$ let $n_k=1$ and if $|k|\geq M$ let
$$
n_k = \left\{
\begin{array}{cl}
w_k &\mbox{if } \vert s_k+t_k-w_k-s_{k+1}\vert<\frac{1}{4},\\
w_k-1 &\mbox{if } \vert 1+s_k+t_k-w_k-s_{k+1}\vert<\frac{1}{4},\\
w_k+1 &\mbox{if } \vert 1+s_{k+1}-s_k-t_k+w_{k}\vert<\frac{1}{4}.
\end{array}\right.
$$
Lemmas 2.4 and 2.5 in \cite{Th} assure that the number $n_k$ is exactly the number satisfying $$\vert s_k+t_k-n_k-s_{k+1}\vert<\frac{1}{4} \,\,\,\,\,\, \text{and}$$ $$d(f^{n_k}(x_k),x_{k+1})<\frac{1}{4}  \,\,\,\,\,\, \text{whenever}  \,\,\,\,\,\, |k|\geq M.$$ 
Consider the sequence $(y_i)_{i\in\mathbb{Z}}\subset X$ defined as follows: for each $i\in\Z$ let
	\[y_i=
	f^{i-N_k}(x_k) \text{ if }\ N_k\leq i<N_{k+1}\]
	where $(N_k)_{k\in \mathbb{Z}}$ is the sequence of sums associated to $(n_k)_{k\in \mathbb{Z}}$. We are going to prove that $(y_i)_{i\in\Z}$ is a two-sided limit pseudo orbit of $f$. It is enough to prove that $$d(f^{n_k}(x_k),\,x_{k+1})\to0, \,\,\,\,\,\, \text{when} \,\,\,\,\,\, |k|\to\infty.$$ For each $0<\eps<\frac{1}{4}$, choose $\eps'>0$, given by Lemma 2.5 in \cite{Th},
and choose $N>0$ such that $$d^{f}((f^{w_k}(x_k),s_k+t_k-w_k),\, (x_{k+1},s_{k+1}))<\eps' \,\,\,\,\,\, \text{whenever} \,\,\,\,\,\, |k|\geq N.$$ The choice of $\eps'$ assures that 
$$\vert s_k+t_k-n_k-s_{k+1}\vert<\eps' \,\,\,\,\,\, \text{and}$$ $$d(f^{n_k}(x_k),x_{k+1})<\eps  \,\,\,\,\,\, \text{whenever}  \,\,\,\,\,\, |k|\geq N.$$ Since this can be done for each $\eps>0$, the claim follows and $(y_i)_{i\in\Z}$ is a two-sided limit pseudo orbit of $f$. Since $f$ has the two-sided limit shadowing property, there exists $x\in X$ satisfying $$\lim_{|i|\to+\infty}d(f^i(x),\,y_i)=0.$$  We claim that $(x,\,s_0)$ two-sided limit shadows $((x_k,\,s_k),\,t_k)_{k\in\mathbb{Z}}$. To prove this, we define a reparametrization $\alpha\in\text{Rep}$ and analyze for each $t\in\R$ the distance in the metric $d^f$ between the points $\phi^f_{\alpha(t)}(x,s_0)$ and $\phi^f_{t-T_k}(x_k,s_k)$, where $(T_k)_{k\in \mathbb{Z}}$ denotes the sequence of sums associated to $(t_k)_{k\in \mathbb{Z}}$. Define $\alpha\colon\mathbb{R}\to\mathbb{R}$  as follows:
	\[\alpha(t)=
	\frac{s_{k+1}+n_k-s_k}{t_k}(t-T_k)+s_k+N_k-s_0, \,\, \text{ if } \,\,\, T_k\leq t<T_{k+1}.
	\]
It is easy to check that $\alpha(0)=0$, $\alpha$ is piecewise linear and continuous. Also, $s_{k+1}+n_k-s_k>0$, since $n_k\geq 1$, so $\alpha$ is increasing and, hence, $\alpha\in \mathrm{Rep}$. If $k\in\Z$ and $T_k\leq t<T_{k+1}$, we choose $j\in\N$ such that $0\leq s_k+t-T_k-j<1$, and write 
$$\phi^f_{\alpha(t)}(x,s_0)=(f^{j+N_k}(x),s_0+\alpha(t)-N_k-j) \,\,\,\,\,\, \text{and}$$ $$\phi^f_{t-T_k}(x_k,s_k)=(f^j(x_k),s_k+t-T_k-j).$$
We will see below that it is enough to estimate the distance between both time and space coordinates of these points. For the time coordinates we proceed as follows: for each $\eps>0$, we choose $0<\eps'<\frac{\eps}{2}$ and $N>0$ (as before) such that $$\vert s_k+t_k-n_k-s_{k+1}\vert<\eps' \,\,\,\,\,\, \text{whenever}  \,\,\,\,\,\, |k|\geq N$$ and using the definition of $\alpha$ we obtain
	\begin{eqnarray*}
		\vert s_0+\alpha(t)-N_k-s_k-t+T_k\vert &=& \vert \alpha(t)-s_k-N_k+s_0-(t-T_k)\vert \\
		&=&\left\vert \frac{s_{k+1}+n_k-s_k-t_k}{t_k}(t-T_k)\right\vert\\
		&=& \vert s_{k+1}+n_k-s_k-t_k\vert\left\vert\frac{t-T_k}{t_k}\right\vert\\
		&<&\eps',
	\end{eqnarray*}
where the last inequality is assured by $0\leq t-T_k<t_k$. For the space coordinates we note that $$\lim_{|k|\to\infty}d(f^{N_k}(x),x_k)=0$$ since $x$ two-sided limit shadows $(y_i)_{i\in\Z}$ and that $j<5$ since $$s_k\in[0,1) \,\,\,\,\,\, \text{and} \,\,\,\,\,\, 0\leq t-T_k<t_k<4.$$ Using uniform continuity of $f$ one can prove that
$$\lim_{|k|\to\infty}d(f^{i+N_k}(x),f^i(x_k))=0, \,\,\,\,\,\, \text{for every} \,\,\,\,\,\, 0\leq i\leq5.$$ Then we can assume that $N$ was chosen such that 
$$d(f^{i+N_k}(x),f^i(x_k))<\frac{\eps}{2}  \,\,\,\,\,\, \text{whenever} \,\,\,\,\,\, 0\leq i\leq5 \,\,\,\,\,\, \text{and} \,\,\,\,\,\, |k|\geq N.$$ This holds, in particular, for $i=j$ and $i=j+1$.
After estimating the difference between both time and space coordinates of the points $\phi^f_{\alpha(t)}(x,s_0)$ and $\phi^f_{t-T_k}(x_k,s_k)$ we can prove they are $\eps$-close in the distance $d^f$. Indeed, $d^{f}(\phi^{f}_{\alpha(t)}(x,\,s_0),\phi^f_{t-T_k}(x_k,s_k))$ is equal to
$$d^{f}\left((f^{j+N_k}(x),s_0+\alpha(t)-N_k-j),\,(f^j(x_k),s_k+t-T_k-j)\right)$$ and with a triangle inequality we obtain that this is
\begin{eqnarray*}
		&\leq& d^{f}\left((f^{j+N_k}(x),\,s_0+\alpha(t)-N_k-j),\,(f^{j+N_k}(x_k),s_k+t-T_k-j)\right)\\
		&+&d^{f}\left((f^{j+N_k}(x),s_k+t-T_k-j),\,(f^{j}(x_k),\,s_k+t-T_k-j)\right).
\end{eqnarray*}
In the first term, the space coordinates are the same, so $d^f$ is just the difference between its time coordinates, while in the second term, the time coordinates are the same and $d^f$ is the distance in the level $s_k+t-T_k-j$. Then we obtain that 
\begin{eqnarray*}
		d^{f}(\phi^{f}_{\alpha(t)}(x,\,s_0),\phi^f_{t-T_k}(x_k,s_k))&\leq & \vert s_0+\alpha(t)-N_k-j-(s_k+t-T_k-j)\vert+\\
		& &+(s_k+t-T_k-j)d(f^{j+N_k+1}(x),\,f^{j+1}(x_k)) +\\
		& & +(1-s_k-t+T_k+j)d(f^{j+N_k}(x),f^j(x_k))\\
		&< & \eps^{\prime}+\tfrac{\eps}{2}(1-(s_k+t-T_k-j))+\tfrac{\eps}{2}(s_k+t-T_k-j)\\
		&\leq& \eps.
	\end{eqnarray*}
Thus, for each $\eps>0$ we found $N\in\N$ such that 
$$d^{f}(\phi^{f}_{\alpha(t)}(x,\,s_0),\phi^f_{t-T_k}(x_k,s_k))<\eps \,\,\,\,\,\, \text{whenever} \,\,\,\,\,\, T_k\leq t<T_{k+1} \,\,\,\,\,\, \text{and} \,\,\,\,\,\, |k|\geq N.$$ It follows that
	\[
	\lim_{|t|\to\infty}d^{f}(\phi^{f}_{\alpha(t)}(x,\,s_0),\phi^f_{t-T_k}(x_k,s_k))=0
	\]
	and, hence, $(x,\,s_0)$ two-sided limit shadow $((x_k,\,s_k),\,t_k)_{k\in\mathbb{Z}}$.




\end{proof}



\vspace{+0.4cm}

\section{Two-sided limit shadowing with a gap}

In this section, we prove Theorems \ref{gap} and \ref{equivalent} regarding the two-sided limit shadowing property with a gap. We begin with the definition of this property for homeomorphisms as in \cite{CK}. We say that a sequence $(x_i)_{i\in\Z}$ is \emph{two-sided limit shadowed with gap $K\in\Z$} if there exists
a point $y\in X$ satisfying
\begin{align*}
d(f^i(y),x_i)\to 0, \,\,\,\,\,\, i\to-\infty \,\,\,\,\,\, \text{and}\\
d(f^{K+i}(y),x_i)\to 0 \,\,\,\,\,\, i\to\infty.
\end{align*}
For $N\in\N_0$ we say that $f$ has the \emph{two-sided limit shadowing property with gap $N$} if every two-sided limit pseudo-orbit of $f$ is two sided limit shadowed with gap $K\in \Z$ with $|K|\le N$. We also say that $f$ has the \emph{two-sided limit shadowing property with a gap} if such an $N\in\N$ exists.

We prove that the two-sided limit shadowing property with a gap in the base homeomorphism is a necessary condition for its suspension flow admit the two-sided limit shadowing property.

\begin{theorem}\label{gapbase}
If $\phi^{f,\tau}$ has the two-sided limit shadowing property, then $f$ has the two-sided limit shadowing property with a gap.
\end{theorem}
\begin{proof}
	As in the previous proof, we can assume that $\tau\equiv 1$ and $\phi^f$ has the two-sided limit shadowing property. Let $(x_n)_{n\in\mathbb{Z}}$ be a two-sided limit pseudo-orbit of $f$. We claim that $((x_n,\frac{1}{2}),1)$ is a two-sided limit pseudo-orbit of $\phi^f$. First, note that $$\phi^f_1(x_n,\tfrac{1}{2})=(f(x_n),\tfrac{1}{2}), \,\,\,\forall n\in\N.$$ Now, for each $\eps>0$, choose $0<\delta<\eps$, given by uniform continuity of $f$, such that
	\[
	d(x,y)<\delta \,\,\, \Longrightarrow \,\,\, d(f(x),f(y))<\eps,\ \forall\,x,y\in X.
	\]
Since $(x_n)_{n\in\mathbb{Z}}$ is a two-sided limit pseudo-orbit, there exists $N\in\N$ such that $$d(f(x_n),x_{n+1})<\delta \,\,\,\,\,\, \text{whenever} \,\,\,\,\,\, |n|\geq N.$$ It follows, in this case, that 
	\begin{align*}
	d^f(\phi^f_1(x_n,\tfrac{1}{2}),(x_{n+1},\tfrac{1}{2}))&=d^f((f(x_n),\tfrac{1}{2}),(x_{n+1},\tfrac{1}{2}))\\
	&=\tfrac{1}{2}d(f(x_n),x_{n+1})+\tfrac{1}{2}d(f^2(x_n),f(x_{n+1}))\\
	&<\tfrac{1}{2}\delta+\tfrac{1}{2}\varepsilon=\varepsilon.
	\end{align*}
This proves the claim and the two-sided limit shadowing property of $\phi^f$ assures the existence of $(x,s)\in X^{1,f}$ and $h\in \mathrm{Rep}$ satisfying 
	\[
	\lim_{|t|\to+\infty}d^f((x_0,\tfrac{1}{2})\star t,\phi^f_{h(t)}(x,s))=0.
	\]
We will prove that some point in the orbit of $x$ two-sided limit shadows $(x_n)_{n\in\mathbb{Z}}$ with some gap. The idea is that after some time, the reparametrized orbit of $(x,s)$  is close to segments of orbit of time 1 for the suspension flow, so it behaves similarly to the suspension flow. This time will give us the gap in the definition of two-sided limit shadowing with a gap and after it the orbit of $x$ by the base homeomorphism $f$ will follow the pseudo orbit $(x_n)_{n\in\Z}$. To see this,  choose $T>0$ such that 
$$d^f((x_0,\tfrac{1}{2})\star t,\phi_{h(t)}^f(x,s))<\frac{1}{4} \,\,\,\,\,\, \text{whenever} \,\,\,\,\,\, |t|\geq T.$$
Let $M=\lfloor T+1\rfloor$ be the integer part of T+1 and put $t=M$ and $t=-M$ in the last inequality to obtain 
$$d^f((x_M,\tfrac{1}{2}),\phi^f_{h(M)}(x,s))<\frac{1}{4} \,\,\,\,\,\, \text{and} \,\,\,\,\,\, d^f((x_{-M},\tfrac{1}{2}),\phi^f_{h(-M)}(x,s))<\frac{1}{4}.$$ Let $N_1=\lfloor h(M)+s\rfloor$, $N_2=\lfloor h(-M)+s\rfloor$ and write 
$$\phi^f_{h(M)}(x,s)=(f^{N_1}(x),h(M)+s-N_1) \,\,\,\,\,\, \text{and}$$ 
$$\phi^f_{h(-M)}(x,s)=(f^{N_2}(x),h(-M)+s-N_2).$$
It follows that,
$$d^f((x_M,\tfrac{1}{2}),(f^{N_1}(x),h(M)+s-N_1))<\frac{1}{4} \,\,\,\,\,\, \text{and}$$
$$d^f((x_M,\tfrac{1}{2}),(f^{N_2}(x),h(-M)+s-N_2))<\frac{1}{4}.$$	
Since 
$$d^f(\phi_{t-M}(x_M,\tfrac{1}{2}),\phi^f_{h(t)}(x,s))<\frac{1}{4} \,\,\,\,\,\, \text{whenever} \,\,\,\,\,\, M\leq t<M+1 \,\,\,\,\,\, \text{and}$$ $$d^f(\phi_{t+M}(x_M,\tfrac{1}{2}),\phi^f_{h(t)}(x,s))<\frac{1}{4} \,\,\,\,\,\, \text{whenever} \,\,\,\,\,\, -M\leq t<-M+1$$ we can write $$\phi^f_{h(M+1)}(x,s)=(f^{N_1+1}(x),h(M+1)+s-N_1-1) \,\,\,\,\,\, \text{and}$$
$$\phi^f_{h(-M-1)}(x,s)=(f^{N_2-1}(x),h(-M-1)+s-N_2+1).$$ 
In a similar way, we can conclude that the following hold for every $n\in\N$:
$$\phi^f_{h(M+n)}(x,s)=(f^{N_1+n}(x),h(M+n)+s-N_1-n) \,\,\,\,\,\, \text{and}$$
$$\phi^f_{h(-M-n)}(x,s)=(f^{N_2-n}(x),h(-M-n)+s-N_2+n).$$
Since $(x,s)$ two-sided limit shadows $((x_n,\frac{1}{2}),1)$, we obtain
$$\lim_{n\to\infty}d^f((x_{M+n},\tfrac{1}{2}),(f^{N_1+n}(x),h(M+n)+s-N_1-n))=0 \,\,\,\,\,\, \text{and}$$
$$\lim_{n\to\infty}d^f((x_{-M-n},\tfrac{1}{2}),(f^{N_2-n}(x),h(-M-n)+s-N_2+n))=0.$$
These imply, in particular, that $$\lim_{n\to\infty}d(x_{M+n}, f^{n+N_1}(x))=0 \,\,\,\,\,\, \text{and} \,\,\,\,\,\, \lim_{n\to\infty}d(x_{-M-n}, f^{N_2-n}(x))=0.$$
Indeed, for each $\varepsilon>0$ choose $0<\varepsilon'<\frac{1}{4}$ as in Lemma 2.5 of \cite{Th} and $N\in\N$ such that 
$$d^f((x_{M+n},\tfrac{1}{2}),(f^{N_1+n}(x),h(M+n)+s-N_1-n))<\varepsilon' \,\,\,\,\,\, \text{and}$$ 
$$d^f((x_{-M-n},\tfrac{1}{2}),(f^{N_2-n}(x),h(-M-n)+s-N_2+n))<\varepsilon'$$
whenever $|n|\geq N$.
It follows that
$$|h(M+n)+s-N_1-n-\tfrac{1}{2}|<\varepsilon' \,\,\,\,\,\, \text{and} \,\,\,\,\,\, |h(-M-n)+s-N_2+n-\tfrac{1}{2}|<\varepsilon'$$
so Lemma 2.5 in \cite{Th} assures that
$$d(x_{M+n},f^{N_1+n}(x))<\varepsilon \,\,\,\,\,\, \text{and} \,\,\,\,\,\, d(x_{-M-n},f^{N_2-n}(x))<\varepsilon.$$ To conclude, observe that
 $$\lim_{n\to\infty}d(x_{n}, f^{n+N_1-M}(x))=0 \,\,\,\,\,\, \text{and} \,\,\,\,\,\, \lim_{n\to\infty}d(x_{-n}, f^{N_2-n+M}(x))=0$$ and that this imply that $f^{N_2+M}(x)$ two-sided limit shadows $(x_n)_{n\in\Z}$ with gap $N_1-2M-N_2$.
\end{proof}


Now we turn our attention to the two-sided limit shadowing property with a gap for continuous flows, as defined in the introduction. We prove that the two-sided limit shadowing property is equivalent to the two-sided limit shadowing property with a gap in this case. The idea is that when a shadowing orbit has a gap in its shadowing relation we can reparametrize the flow to remove this gap. The following is a technical lemma that we use to define the above mentioned reparametrization

\begin{lemma}\label{conserta}
If $z\in X$, $K\in\R$ and $h\in\text{Rep}$, then there exists $\alpha\in\text{Rep}$ such that $\alpha(t)=h(t)$ for every $t\leq0$ and
$$d(\phi_{h(t)+K}(z), \phi_{\alpha(t)}(z))\to0 \,\,\,\,\,\, \text{when} \,\,\,\,\,\, t\to\infty.$$
\end{lemma}

\begin{proof}
We split the proof in the cases $K>0$ and $K<0$. In the first case, define $\alpha(t)=h(t)+k(t)$, where $k\colon\R\to\R$ is defined by $$k(t)=\begin{cases}0& t\leq0\\
Ke^{-\frac{1}{t^2}}& t> 0.\end{cases}$$

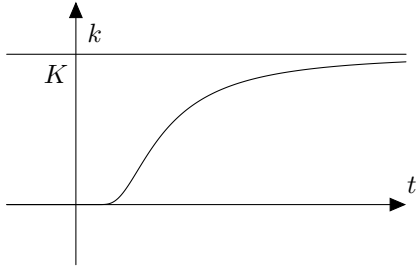
\begin{center}
	\begin{tikzpicture}[line cap=round,line join=round,>=triangle 45,x=1.0cm,y=1.0cm]
	\draw[->,color=black] (-0.92,0) -- (4.39,0);
	\foreach \x in {-0.5,0.5,1,1.5,2,2.5,3,3.5,4}
	\draw[shift={(\x,0)},color=black] (0pt,-2pt);
	\draw[color=black] (4.28,0.03) node [anchor=south west] { $t$};
	\draw[->,color=black] (0,-0.8) -- (0,2.7);
	
	\draw[color=black] (0.03,2.04) node [anchor= south west] { $k$};
	\clip(-0.92,-1.5) rectangle (4.39,2.17);
	\draw [domain=-0.92:4.39] plot(\x,{(--2-0*\x)/1});
	\draw[smooth,samples=100,domain=-0.9182716975672461:-2.0492660421095513E-6] plot(\x,{0});
	\draw[smooth,samples=100,domain=0.009:4.388909088342245] plot(\x,{2*2.718281828^(-1/(\x)^2)});
	\draw (0.02,1.98) node[anchor=north east] {$K$};
	\end{tikzpicture}
\end{center}
Note that $\alpha(t)=h(t)$ for every $t\leq0$ since $k(t)=0$ in this case. Also, $\alpha\in\text{Rep}$ since $h\in\text{Rep}$, $k(0)=0$, $k$ is continuous and increasing in its positive part. Moreover, $k(t)$ converge to $K$, when $t\to+\infty$, since $e^{-\frac{1}{t^2}}$ converge to $1$, so the convergence in the lemma follows from the continuity of $\phi$. In the second case, we use the fact that $h\in\text{Rep}$ to obtain $t_0>0$ such that $h(t_0)+K=1$ and define $\alpha\colon\R\to\R$ by 
$$\alpha(t)=\begin{cases}h(t),& t\leq0\\
\frac{t}{t_0},& 0<t\leq t_0,\\
h(t)+K,& t>t_0.\end{cases}$$
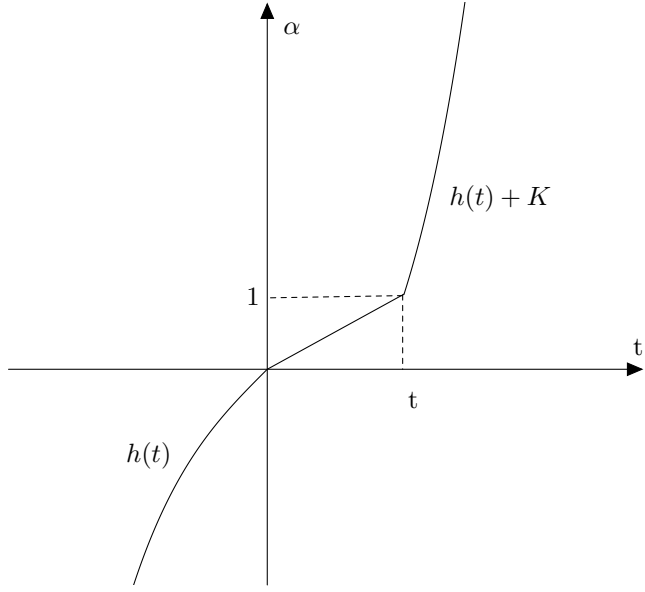
\begin{center}
		\begin{tikzpicture}[line cap=round,line join=round,>=triangle 45,x=1.0cm,y=1.0cm]
	\draw[->,color=black] (-3.44,0) -- (5,0);
	\foreach \x in {-3,-2,-1,1,2,3,4,5}
	\draw[shift={(\x,0)},color=black] (0pt,-2pt);
	\draw[color=black] (4.74,0.07) node [anchor=south west] { t};
	\draw[->,color=black] (0,-2.87) -- (0,4.88);
	
	\draw[color=black] (0.09,4.53) node [anchor=west] { $\alpha$};
	\clip(-3.44,-2.87) rectangle (5,4.88);
	\draw[smooth,samples=100,domain=-3.436273585928:-3.826350615457983E-6] plot(\x,{(2.718281828^(\x)-2.718281828^(-(\x)))/2});
	\draw[smooth,samples=100,domain=1.8184472545264312:5.002523561357201] plot(\x,{((2.718281828^(\x)-2.718281828^(-(\x)))/2-2)});
	\draw[smooth,samples=100,domain=4.612446531827218E-6:1.818445262970305] plot(\x,{(\x)/ln(3+sqrt(10))});
	\draw [dash pattern=on 2pt off 2pt] (1.8,0.99)-- (1.8,0);
	\draw [dash pattern=on 2pt off 2pt] (1.8,0.98)-- (0,0.95);
	\draw (1.75,-0.18) node[anchor=north west] {t};
	\draw (-0.4,1.2) node[anchor=north west] {1};
	\draw (-2,-0.85) node[anchor=north west] {$h(t)$};
	\draw (2.3,2.58) node[anchor=north west] {$h(t)+K$};
	\end{tikzpicture}
\end{center}

The choice of $t_0$ assures that $\alpha\in\text{Rep}$ and the other properties of the lemma follow directly from the definition of $\alpha$.
\end{proof}


\begin{proof}[Proof of Theorem \ref{equivalent}]
Suppose that $\phi$ has the two-sided limit shadowing property with a gap. Let $(x_k,t_k)_{k\in\Z}$ be a two-sided limit pseudo orbit of $\phi$ and suppose that $z\in X$, $K\in\R$ and $h\in\text{Rep}$ satisfy 
$$d(x_0\star t, \phi_{h(t)}(z))\to0, \,\,\,\,\,\, t\to-\infty \,\,\,\,\,\, \text{and}$$
$$d(x_0\star t, \phi_{h(t)+K}(z))\to0, \,\,\,\,\,\, t\to+\infty.$$
Consider $\alpha\in\text{Rep}$ given by the previous lemma. Then $$d(x_0\star t, \phi_{\alpha(t)}(z))\to0, \,\,\,\,\,\, t\to-\infty$$ since $\alpha(t)=h(t)$ for $t\leq0$, while for $t>0$ it follows that
\begin{align*}
d(x_0\star t,\phi_{\alpha(t)}(z))\leq d(x_0\star t,\phi_{h(t)+K}(z))+d(\phi_{h(t)+K}(z),\phi_{\alpha(t)}(z))
\end{align*}
where these two terms converge to zero when $t\to\infty$.
\end{proof}

\begin{remark}
To finish the proof of Theorem \ref{gap} one just need to prove that the suspension of a homeomorphism satisfying the two-sided limit shadowing property with a gap also satisfies the two-sided limit shadowing property with a gap. The calculations are very similar to the proof of Theorem \ref{sus} for the case with no gaps and we leave it as an exercise.
\end{remark}

\vspace{+0.4cm}

\section{Shadowing and transitivity}

In this section, we prove Theorem \ref{transha} where transitivity and the finite shadowing property are proved assuming two-sided limit shadowing. The arguments follow the arguments in the case of homeomorphisms in \cite{CK} with suitable adaptations. We recall necessary definitions. A flow is said to be {\em transitive}, if there is $x\in X$ whose omega limit set
$$\omega(x):=\{y\in X\colon y=\lim_{t_k\to+\infty}\phi_{t_k}(x)\ \text{for some sequence }t_k\to +\infty\}$$ 
equals the whole space. To define the shadowing property, let $\eps>0$ and $\delta>0$ be given and consider a sequence of pairs $(x_n,t_n)_{n\in\mathbb{Z}}\subset X\times\R^+$. This sequence is called a $\delta$-pseudo-orbit of $\phi$ if $t_n\geq1$ and 
$$d(\phi_{t_n}(x_n),x_{n+1})\leq\delta\,\,\,\,\,\, \text{for every} \,\,\,\,\,\,n\in\mathbb{Z}.$$
We say that this sequence is $\eps$-shadowed if there exist $h\in\mathrm{Rep}$ and $y\in X$ satisfying
$$d(x_0\star t,\phi_{h(t)})\leq \eps,\ \forall\,t\in\mathbb{R}.$$
The flow $\phi$ has the {\em shadowing property} if for each $\varepsilon>0$ there exists $\delta>0$ such that every $\delta$-pseudo-orbit is $\varepsilon$-shadowed. As in \cite{K} we say that $\phi$ has the \emph{finite shadowing property} if for each $\varepsilon>0$, there exists $\delta>0$ such that every finite $\delta$-pseudo-orbit is $\varepsilon$-shadowed. Finally, the flow is called \emph{chain-transitive} if for each pair of points $x,y\in X$ and each $\eps>0$, there exists a finite $\eps$-pseudo orbit starting at $x$ and ending at $y$.

\begin{lemma}\label{fshadowing}
If $\phi$ is chain transitive and has the limit shadowing property, then it has the finite shadowing property.
\end{lemma}

\begin{proof}
If this is not the case, then there exists $\eps>0$ such that for each $n\in\mathbb{N}$ there is a finite $\frac{1}{n}$-pseudo-orbit $\alpha_n$  that is not $\eps$-shadowed. By chain transitivity, for each $n$ there is a finite $\frac{1}{n}$-pseudo orbit $\beta_n$ such that the concatenated sequence $\alpha_n\beta_n\alpha_{n+1}$ forms a finite $\frac{1}{n}$-pseudo-orbit of $\phi$. Thus, the sequence
$$\alpha_1\beta_1\alpha_2\beta_2\alpha_3\beta_3\cdots$$
is a limit pseudo orbit of $\phi$. Write this pseudo-orbit as $(x_i,\,t_i)_{i\in\mathbb{N}}\subset X\times\R$. The limit shadowing property assures the existence of $h\in\mathrm{Rep}$ and $q\in X$ satisfying
$$\lim_{t\to+\infty}d(x_0\star t,\,\phi_{h(t)}(q)) = 0.$$
Let $(s_n)_{n\in\N}$ be the sequence of sums associated to $(t_n)_{n\in\N}$ and choose $N\in\mathbb{N}$ such that 
$$d(\phi_{t-s_n}(x_n),\,\phi_{h(t)}(q))\leq\eps\,\,\,\,\,\, \text{whenever} \,\,\,\,\,\, s_n\leq t < s_{n+1} \,\,\,\,\,\, \text{and} \,\,\,\,\,\, n\geq N.$$ 
For each $n\geq N$, let $r_n=s_n-s_N$ and define $g\in\text{Rep}$ as 
$$g(t)=h(t+s_N)-h(s_N) \,\,\,\,\,\, \text{whenever} \,\,\,\,\,\, r_n\leq t<r_{n+1} \,\,\,\,\,\, \text{and} \,\,\,\,\,\, n\in\N.$$ 
Then $(x_n,\,t_n)_{n=N}^\infty$ can be $\eps$-shadowed by $\phi_{h(s_N)}(q)$ with the reparametrization $g$. Indeed, if $n\geq N$ and $r_n\leq t<r_{n+1}$, then $s_n\leq t+s_N<s_{n+1}$ and
$$d(x_N\star t,\,\phi_{g(t)}(\phi_{h(s_N)}(q))) = d(x_0\star(t+s_N),\,\phi_{h(t+s_N)}(q))\leq\eps.$$
But this implies that some $\alpha_n$ can be $\eps$-shadowed contradicting the assumption.

\end{proof}

The following is the analogue of Lemma 3.2 in \cite{CK} in the case of flows.

\begin{lemma}\label{chain}
If a continuous flow has the two-sided limit shadowing property, then it is chain transitive.	
\end{lemma}
\begin{proof}
For each $x,y\in X$ and $\eps>0$ we will prove the existence of an $\eps$-pseudo orbit starting at $x$ and ending at $y$. Choose $z_1\in\omega(x)$ and $z_2\in\alpha(y)$ and consider the sequence $(x_n,t_n)_{n\in\mathbb{Z}}$ defined as follows:
	\[
	(x_n,t_n)=
	\begin{cases}
		(\phi_{n}(z_2),1) &\text{if }n\geq0,\\
		(\phi_{n}(z_1),1)&\text{if }n<0.
	\end{cases}
	\]
It is clear that $(x_n,t_n)_{n\in\mathbb{Z}}$ is a two-sided limit pseudo-orbit of $\phi$ and, hence, there are $z\in X$ and $h\in\mathrm{Rep}$ such that 
	\[
	\lim_{|t|\to\infty}d(x_0\star t,\phi_{h(t)}(z))=0.
	\]
In particular, there exists $T\geq 1$ such that if $t\geq T$, then
$$d(\phi_{-t}(z_1),\phi_{h(-t)}(z))<\frac{\eps}{2}\ \text{ and }\ d(\phi_t(z_2),\phi_{h(t)}(z))<\frac{\eps}{2}.$$
Suppose that $T$ was also chosen satisfying $h(T)-h(-T)\geq1$.
Since $\omega(x)$ and $\alpha(y)$ are invariant under $\phi$ it follows that $\phi_{-T}(z_1)\in \omega(x)$ and $\phi_{T}(z_2)\in \alpha(y)$.  Choose $T'>T$ such that
	  \[d(\phi_{-T}(z_1),\phi_{T'}(x))<\frac{\eps}{2}\  \text{ and }\  d(\phi_{T}(z_2),\phi_{-T'}(y))<\frac{\eps}{2}.\]
Therefore
	   \[d(\phi_{T'}(x),\phi_{h(-T)}(z))<\eps\  \text{ and }\  d(\phi_{-T'}(y),\phi_{h(T)}(z))<\eps.\]
Now we define a $\eps$-pseudo orbit from $x$ to $y$ following the orbit of $x$ till $\phi_{T'}(x)$, then the orbit of $\phi_{h(-T)}(z)$ up to $\phi_{h(T)}(z)$ and then the orbit of $\phi_{-T'}(y)$ until $y$. This finishes the proof.
\end{proof}

The singular suspension flow of a homeomorphism with the shadowing property is an example of a continuous flow with a fixed point, satisfying the finite shadowing property but not the shadowing property (see \cite{K}). It is a classical result that flows with the shadowing property are chain-transitive if, and only if, they are transitive. Since the two-sided limit shadowing property implies just the finite shadowing property, transitivity does not follow immediately Instead, we prove transitivity using the limit shadowing property as in the next result.

\begin{lemma}\label{chain2}
If a continuous flow has the limit shadowing property, then it is chain transitive if, and only if, it is transitive.
\end{lemma}
\begin{proof}
It is enough to assume chain-transitivity and prove transitivity. For each $k\in\N$, we can use compactness of $X$ to choose $n_k$ points $x_1^k,\, x^k_2,\,\dots,\,x^k_{n_k}$ such that 
	\[
	X=\bigcup_{i=1}^{n_k}B\left(x^k_{i},1/2^k\right).
	\]
Since $\phi$ is chain-transitive, we can choose, for each $k\in\N$ and $i\in\{1,\,2,\,\dots,\,n_k-1\}$, a $\frac{1}{2^k}$-pseudo-orbit $\alpha_i^k$ from $x^k_i$ to $x^k_{i+1}$ and also a $\frac{1}{2^k}$-pseudo-orbit $\beta_k$ from $x^k_{n_k}$ to $x^{k+1}_1$. Thus, the concatenated sequence 
	\[	\alpha_1^1\alpha_2^1\cdots\alpha_{n_1-1}^1\beta_1\alpha_1^2\alpha_2^2\cdots\alpha_{n_2-1}^2\beta_2\cdots\alpha_1^k\alpha_2^k\cdots\alpha_{n_k-1}^k\beta_k\cdots
	\]
	is a limit pseudo-orbit of $\phi$, so it can be limit shadowed by $y$ with a reparametrization $h\in\mathrm{Rep}$. We will prove that $\omega(y)=X$ and this proves transitivity of $\phi$. Write the above limit pseudo orbit as $(y_n,t_n)_{n\in\mathbb{N}\cup\{0\}}$ with $t_n\geq1$ and let $(s_n)_{n\in\mathbb{N}\cup\{0\}}$ be the corresponding sequence of sums of $(t_n)_{n\in\mathbb{N}\cup\{0\}}$. For each $z\in X$ and $\eps>0$, we will prove the existence of a point in the future orbit of $y$ in $B(z,\eps)$. 
Choose $T>0$ such that
$$d(y_0\star t,\phi_{h(t)}(y))\leq \frac{\varepsilon}{2} \,\,\,\,\,\, \text{for every} \,\,\,\,\,\, t\geq T$$
and $k>0$ such that $\frac{1}{2^k}<\frac{\varepsilon}{2}$. Increasing $k$, if necessary, we can choose $$x\in\{x^k_1,\,\cdots,\,x^k_{n_k}\}\cap B(z,\varepsilon/2)$$ such that $x=y_m$ for some $m\in\N$ and $s_m\geq T$. Thus, 
	\[
	d(x,\phi_{h(s_m)}(y))=d(\phi_{s_m-s_m}(y_m),\phi_{h(s_m)}(y))\leq\frac{\varepsilon}{2}
	\]
and, hence, 
\begin{eqnarray*}
	d(z,\phi_{h(s_m)}(y))&\leq& d(z,x)+d(x,\phi_{h(s_m)}(y))\\
&\leq& \frac{\varepsilon}{2}+\frac{\varepsilon}{2}=\varepsilon.
\end{eqnarray*}
	This finishes the proof.
\end{proof}

\begin{proof}[Proof of Theorem \ref{transha}]
Suppose that $\phi$ has the two-sided limit shadowing property. Theorem \ref{chain} assures that $\phi$ is chain transitive and since it has the limit shadowing property, Lemma \ref{chain2} assures it is transitive. The finite shadowing property is proved in Lemma \ref{fshadowing}. In the case $\phi$ does not have fixed points, Theorem 4 in \cite{K} assures that $\phi$ has the shadowing property.
\end{proof}



\vspace{+0.4cm}

\section{Singular Suspensions}\label{singu}

In this section we prove Theorem \ref{singularsus}. We begin stating the definition of a singular suspension as in \cite{K}. Let $X$ be a compact subspace of $\mathbb{R}^n$ with $\diam(X)\leq 1$ with respect to the Euclidean metric $|.|$ and $f\colon X\to X$ be a homeomorphism. Consider $X^f$ the suspension space of $f$ under the constant function $1$, choose $a\in X$ and set $e=(a,\frac{1}{2})$. Let
	$$U=\left\{x\in\mathbb{R}^{n+1}; \,\,\,\left|e-x\right|<\frac{1}{4}\right\}$$
	and $c:\mathbb{R}^{n+1}\to [0,1]$ be any $C^\infty$ function satisfying
\begin{enumerate}
\item $c(x)=0$ if, and only if, $x=e$; 
\item $0\leq c(x)<1$, if $x\in U$;
\item $c(x)=1$ if, and only if, $x\notin U$.
\end{enumerate}
Let $\varphi$ be the flow defined by the vector field 
$$\dot{x}_1=0, \ \ \ \dot{x}_2=c(x), \text{ for }x=(x_1,x_2)\in \mathbb{R}^{n}\times\mathbb{R}.$$
We restrict $\varphi$ to the set $X\times[0,1]$, call the induced flow on $X^f$ the \emph{singular suspension of} $f$ with singularity $e=(a,\frac{1}{2})$ and denote it by $\varphi^f$. 
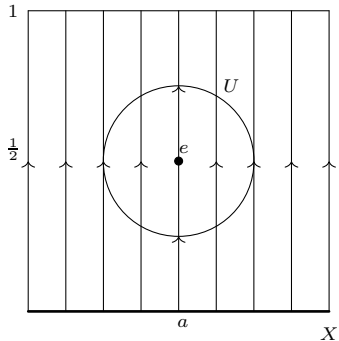
\begin{center}

\begin{tikzpicture}[line cap=round,line join=round,x=1cm,y=1cm]
\clip(-0.6869410699876265,-0.5287005569038019) rectangle (5.9172723191270995,4.206258602031966);
\draw [->] (0,0) -- (0,2);
\draw [->] (0.5,0) -- (0.5,2);
\draw [->] (1,0) -- (1,2);
\draw [->] (1.5,0) -- (1.5,2);
\draw [->] (2,0) -- (2,1);
\draw [->] (2,2) -- (2,3);
\draw [->] (2.5,0) -- (2.5,2);
\draw [->] (3,0) -- (3,2);
\draw [->] (3.5,0) -- (3.5,2);
\draw [->] (4,0) -- (4,2);
\draw  (0,2)-- (0,4);
\draw  (0.5,2)-- (0.5,4);
\draw  (1,2)-- (1,4);
\draw  (1.5,2)-- (1.5,4);
\draw (2.5,2)-- (2.5,4);
\draw  (3,2)-- (3,4);
\draw (3.5,2)-- (3.5,4);
\draw  (4,2)-- (4,4);
\draw  (0,4)-- (4,4);
\draw [line width=1pt] (0,0)-- (4,0);
\draw  (2,1)-- (2,2);
\draw (2,3)-- (2,4);
\begin{scriptsize}

\draw (-0.2,4) node {$1$};

\draw (2.0551166917923864,-0.15863222423203532) node {$a$};
\draw (-0.2,2.158806984151138) node {$\frac{1}{2}$};
\draw [fill=black] (2,2) circle (1.5pt);
\draw (2,2) circle (28.5pt);
\draw (2.07,2.158806984151138) node {$e$};
\draw[color=black] (4,-0.3) node {$X$};
\draw[color=black] (2.7,3) node {$U$};

\end{scriptsize}
\end{tikzpicture}
\end{center}
As a first step to prove Theorem \ref{singularsus}, we generalize Theorem \ref{gapbase} to the case of a singular suspension.

\begin{theorem}\label{lkj}
If $\varphi^f$ has the two-sided limit shadowing property, then the base homeomorphism $f$ has the two-sided limit shadowing property with a gap.
\end{theorem}

\begin{proof}
First, we note that $a$ cannot be isolated in $X$. Indeed, Theorem \ref{transha} assures that $\varphi^f$ is transitive, which, in turn, implies that the base homeomorphism $f$ is transitive, since a dense forward orbit for the singular suspension flow is also a dense forward orbit for the usual suspension flow, which is above a dense forward orbit for the base homeomorphism. If $a$ is isolated, then it is necessarily a transitive point of $f$. We note that it also must be periodic, since it belongs to its omega limit set and is isolated. This would prove that the space $X$ is just the periodic orbit of $a$ and that the space $X^f$ is a homoclinic loop of the fixed point $e$. The flow $\varphi^f$ does not have the limit shadowing property since we can consider a limit pseudo orbit with an infinite number of turns in this cycle, with jumps arbitrarily close to $e$ from the stable orbit to the unstable orbit. This limit pseudo orbit cannot be limit shadowed by any orbit of the space, which, indeed, contains only two orbits: the fixed point and the homoclinic loop.

Now let $(x_n)_{n\in\mathbb{Z}}\subset X$ be a two-sided limit pseudo-orbit of $f$ and define $(y_n)_{n\in\Z}$ as follows: $y_n=x_n$ if $x_n\neq a$ and when $x_n=a$ then let $y_n$ be any point in $X$ satisfying $|y_n-e|<1/n$. It is easy to see that $(y_n)_{n\in\Z}$ is also a two-sided limit pseudo orbit of $f$ and that any point two-sided limit shadowing $(y_n)_{n\in\Z}$ does the same in $(x_n)_{n\in\Z}$. For each $x\in X-\{a\}$ define $\beta(x)>0$ as the smallest real positive number such that $$\varphi^f_{\beta(x)}((x,1/2))=(f(x),1/2).$$ Observe that if one consider the usual suspension flow, then $\beta\equiv1$, but for the singular suspension flow $\beta$ increases arbitrarily as the point approaches $a$.
We claim that $((y_n,\frac{1}{2}),\beta(y_n))_{n\in\Z}$ is a two-sided limit pseudo-orbit of $\varphi^f$. We also observe that this could not be true using $x_n$ instead of $y_n$, since $\beta$ is not defined in $a$.
We have to prove that $$d^f(\varphi^f_{\beta(y_n)}(y_n,1/2), (y_{n+1},1/2))\rightarrow 0 \,\,\,\,\,\, \text{when} \,\,\,\,\,\, n\rightarrow\pm\infty.$$
Note that for each $n\in\N$ we have
\begin{eqnarray*}
d^f(\varphi^f_{\beta(y_n)}(y_n,1/2), (y_{n+1},1/2))&=& d^f((f(y_n),1/2),(y_{n+1},1/2)) \\
&=&\frac{1}{2}|f(y_{n})-y_{n+1})| + \frac{1}{2}|f^2(y_n)-f(y_{n+1})|
\end{eqnarray*}	
so the desired limits follow from the fact that $(y_n)_{n\in\Z}$ is a two-sided limit pseudo orbit and $f$ is uniformly continuous.
Now the proof follows exactly the proof for the usual suspension in Theorem \ref{gapbase} because the shadowing orbit given by the two-sided limit shadowing property of $\varphi^f$ cannot be in the stable or the unstable set of the fixed point, since the segments of orbit of $((y_n,\frac{1}{2}),\beta(y_n))_{n\in\Z}$ intersect the base $X\times \{0\}$ an infinite number of times.

%
%
%
%
%
	%
%

\end{proof}

A second step in the proof of Theorem \ref{singularsus} is done in Proposition \ref{fixedstable} below where we use the two-sided limit shadowing property for the singular suspension flow to prove that morally there is only two possible stable and unstable sets for the whole flow. First, we state and prove two simple facts used in the proof. Since these are general results, we return to the set of the previous sections where $(X,d)$ is a compact metric space and $\phi$ is a continuous flow in $X$. We recall the definition of the stable and unstable sets of a fixed point $p$ of $\phi$, respectively:
$$W^s(p)=\{y\in X; \,\,d(\phi_t(x),p)\to0 \,\,\, \text{when} \,\,\, t\to+\infty\} \,\,\,\,\,\, \text{and}$$
$$W^u(p)=\{y\in X; \,\,d(\phi_t(x),p)\to0 \,\,\, \text{when} \,\,\, t\to-\infty\}.$$

\begin{lemma}\label{sing}
If $p$ is a fixed point of a continuous flow $\phi$ and there exist $x\in X$ and $h\in\text{Rep}$ such that $d(\phi_{h(t)}(x),p)\to0$ when $t\to-\infty$, then $x\in W^u(p)$.
\end{lemma}

\begin{proof}
For each $\eps>0$ choose $t_0<0$ such that $d(\phi_{h(t)}(x),p)<\eps$ for every $t\leq t_0$. Thus, $d(\phi_t(x),p)<\eps$ for every $t\leq h(t_0)$. Since this holds for every $\eps>0$, it follows that $x\in W^u(p)$.
\end{proof}

\begin{remark}
A similar results holds for the stable set $W^s(p)$.
\end{remark}

\begin{lemma}\label{change}
If $g\in Rep$ and $x,y\in X$ satisfy $$d(\phi_{g(t)}(y),\phi_t(x))\to0, \,\,\,\,\,\, t\to+\infty,$$ then $h=g^{-1}\in Rep$ and satisfies $d(\phi_t(y),\phi_{h(t)}(x))\to0$, when $t\to+\infty$.
\end{lemma}

\begin{proof}
For each $t\geq0$ consider $s\geq0$ such that $t=g(s)$. Then one just need to note that
$$d(\phi_t(y),\phi_{h(t)}(x))=d(\phi_{g(s)}(y),\phi_s(x))$$ and that $s\to+\infty$ when $t\to+\infty$.
\end{proof}

\begin{proposition}\label{fixedstable}
If a continuous flow satisfying the two-sided limit shadowing property admits a fixed point $p$ such that $W^u(p)=\{p\}\cup\mathcal{O}(y)$, with $y\neq p$, then for each $x\in X$, it follows that either $x\in W^s(p)$ or there exists $h\in\text{Rep}$ such that $$d(\phi_{h(t)}(x),\phi_t(y))\to0, \,\,\,\,\,\, t\to+\infty.$$
\end{proposition}

\begin{proof}
Consider the two-sided limit pseudo orbit that is formed by the fixed point $p$ in the past and by the future orbit of $x$. The two-sided limit shadowing property assures the existence of $z\in X$ and $\alpha\in\text{Rep}$ such that $z\in W^u(p)$ (Lemma \ref{sing}) and $d(\phi_{\alpha(t)}(z),\phi_t(x))\to0$ when $t\to+\infty$. Since $z\in W^u(p)=\{p\}\cup\mathcal{O}(y)$, it follows that either $z=p$ or $z\in\mathcal{O}(y)$. In the first case, we have $x\in W^s(p)$, while in the second case, $z=\phi_T(y)$ for some $T\in\R$ and we have $$d(\phi_{\alpha(t)+T}(y)),\phi_t(x))\to0, \,\,\,\,\,\, t\to+\infty.$$ Then Lemmas \ref{conserta} and \ref{change} assures the existence of $h\in Rep$ such that $$d(\phi_{h(t)}(x),\phi_t(y))\to0, \,\,\,\,\,\, t\to+\infty.$$ This finishes the proof.
\end{proof}

A similar result can be obtained when $W^s(p)=\{p\}\cup\mathcal{O}(z)$, with $z\neq p$: either $x$ is in the unstable set of the fixed point or it follows the orbit of $z$ in the past. We note that these hypothesis on the stable and unstable set of the fixed point is exactly what happens in the case of the singular suspension flow. To prove Theorem \ref{singularsus} we will need the definition of an equicontinuous homeomorphism. We say that $f$ is \emph{equicontinuous} if for every $\eps>0$, there exists $\delta>0$ such that
$$d(x,y)<\delta \,\,\,\,\,\, \text{implies} \,\,\,\,\,\, d(f^n(x),f^n(y))<\eps \,\,\,\,\,\, \forall n\in\N.$$
The \emph{adding machines} \cite{MY} are examples of transitive and equicontinuous homeomorphisms satisfying the shadowing property (they are indeed minimal, that is, the future orbit of every point of the space is dense in the space).



\begin{proof}[Proof of Theorem \ref{singularsus}]
Suppose that the singular suspension $\varphi^f$ satisfies the two-sided limit shadowing property. Theorem \ref{transha} assures that it is transitive and satisfies the finite shadowing property. Theorem 6 in \cite{K} assures that the base homeomorphism $f$ satisfies the shadowing property and it is also transitive as observed above.
The unstable set of the fixed point is $W^u(e)=\{e\}\cup\mathcal{O}(y)$, with $y\neq e$, and the stable set is $W^s(e)=\{e\}\cup\mathcal{O}(z)$, with $z\neq e$. Let $x\in X$ be a transitive point of $f$ that is not in the past orbit of $a$. Then $(x,0)\notin W^s(e)$ and Proposition \ref{fixedstable} assures the existence of $h\in\text{Rep}$ such that $$d^f(\varphi^f_{h(t)}(x,0),\varphi^f_t(y))\to0, \,\,\,\,\,\, t\to+\infty.$$ It follows from this that $y$ is a transitive point for the singular suspension flow. Proposition \ref{fixedstable} applies again to ensure that every point of $X$ that is not in the past orbit of $a$ is a transitive point of $f$. The same argument applied to $f^{-1}$ proves that every point of $X$ that is not in the future orbit of $a$ is a transitive point for $f^{-1}$.
This implies that the orbit of $a$ for $f$ is infinite. Indeed, if $a$ is periodic, then for each $\eps>0$ we choose $\delta>0$ given by the shadowing property of $f$ and consider a $\delta$-pseudo orbit formed by the past orbit of $a$ and the future orbit of a point $w$ that is $\delta$-close to $a$ and is not in the orbit of $a$ (recall that $a$ is not isolated in $X$).
The shadowing property of $f$ assures the existence of $w'\in X$ that $\eps$-shadows this pseudo-orbit and, hence, $w'\in W^u_{\eps}(a)\cap W^s_{\eps}(w)$. If $w'$ belongs to the orbit of $a$, then the orbit of $w$ is contained in the $\eps$-neighborhood of the orbit of $a$ since $w'\in W^s_{\eps}(w)$. Since $w$ is a transitive point, it follows that $X$ is contained in closed $\eps$-neighborhood of the orbit of $a$. In the case $w'$ does not belong to the orbit of $a$, then $w'$ is a transitive point for $f^{-1}$, as we proved before, and its orbit is entirely contained in the $\eps$-neighborhood of the orbit of $a$. Again, it follows that $X$ is contained in closed $\eps$-neighborhood of the orbit of $a$. Since this can be done for every $\eps>0$, we obtain that $X$ is exactly the orbit of $a$, but this contradicts the fact that $a$ is not isolated in $X$.

The orbit of $a$ being infinite assures that $f(a)$ does not belong to the past orbit of $a$ and, hence, $f(a)$ is transitive. This implies that the points in the past orbit of $a$ are also transitive and we have proved that $f$ is minimal. It follows that $f$ is topologically conjugate to an adding machine map (see \cite{MY}). In particular, $f$ is equicontinuous (and so is $f^{-1}$ as proved in \cite{AG} Theorem 3.4). The stable and the unstable sets of every point of $X$ are trivial. Indeed, if $b\in W^s(c)$ but $|b-c|>0$, we choose $\delta>0$, given by the equicontinuity of $f^{-1}$, and choose $N\in\N$ such that $|f^N(b)-f^N(c)|<\delta$. Then
$$|f^{-n}(f^N(b))-f^{-n}(f^N(c))|<\frac{|b-c|}{2} \,\,\,\,\,\, \text{for every} \,\,\,\,\,\, n\in\N.$$ Letting $n=N$ we obtain $|b-c|<|b-c|/2$ that is a contradiction. Stable and unstable sets being trivial together with the fact that $f$ satisfies the two-sided limit shadowing property with a gap, as proved in Theorem \ref{lkj}, assure that the space $X$ is finite and a single orbit of $f$. Indeed, if $p,q\in X$, then the two-sided limit shadowing property with a gap assures the existence of $z\in W^u(p)\cap W^s(f^m(q))$ for some $m\in\Z$. If $p$ and $q$ belong to different orbits, then $f^m(q)\neq p$ and either $W^u(p)$ or $W^s(f^m(q))$ is non-trivial.
\end{proof}

\begin{remark}
The end of this proof proves that if a minimal homeomorphism satisfies the two-sided limit shadowing property with a gap, then the space is finite and a single periodic orbit. In particular, the adding machines do not satisfy the two-sided limit shadowing property with a gap, even though they satisfy the limit shadowing property (see \cite{Ka}). They satisfy even the orbital two-sided limit shadowing property, as defined in \cite{P1}.
\end{remark}

\vspace{+0.4cm}

\section*{Acknowledgments}
The second author was supported by Capes, CNPq and the Alexander von Humboldt Foundation.

\vspace{1cm}
\noindent

{\em J. Aponte}

\vspace{0.2cm}

\noindent

ESPOL Polytechnic University, 

Escuela Superior Politécnica del Litoral, ESPOL 

Facultad de Ciencias Naturales y Matemáticas 

Campus Gustavo Galindo, Km. 30.5 Vía Perimetral 

P.O. Box 09-01-5863, Guayaquil, Ecuador

\vspace{0.2cm}

\email{japonte@espol.edu.ec}

\vspace{1.5cm}
\noindent

{\em B. Carvalho}
\vspace{0.2cm}

\noindent

Departamento de Matem\'atica,

Universidade Federal de Minas Gerais - UFMG

Av. Ant\^onio Carlos, 6627 - Campus Pampulha

Belo Horizonte - MG, Brazil.

\vspace{0.2cm}
Friedrich-Schiller-Universität Jena

Fakultät für Mathematik und Informatik

Ernst-Abbe-Platz 2

07743 Jena

\vspace{0.2cm}

\email{bmcarvalho@mat.ufmg.br}
\vspace{1.5cm}
\noindent

{\em W. Cordeiro}

\noindent

Institute of Mathematics, Polish Academy of Sciences

ul. \'Sniadeckich, 8

00-656 Warszawa - Poland

\vspace{0.2cm}

\email{wcordeiro@impan.pl}

\noindent

\vspace{0.2cm}
\email{}

\end{document}